\documentclass[amsmath]{article}
\usepackage{latexsym,amssymb,amsthm,amsmath}

\usepackage{amssymb}
\usepackage{eucal}
\usepackage{amsmath}

\theoremstyle{plain}
\newtheorem{theorem}{Theorem}[section]
\newtheorem*{Theorem B}{Theorem B}
\newtheorem*{Theorem A}{Theorem A}
\newtheorem{lemma}{Lemma}[section]

\numberwithin{equation}{section}

\theoremstyle{remark}

\begin{document}
{\it{Accepted Manuscript: Comptes rendus de l'Academie bulgare des Sciences}}\\
\\

\centerline{\Large{A general inequality for contact CR-warped product}}

\smallskip
\centerline{\Large{submanifolds in cosymplectic space forms}}

\bigskip

\centerline{{{Falleh R. Al-Solamy and Siraj Uddin}}}

\begin{center}
{\small Dedicated to the MSU Distinguished Professor Bang-Yen Chen \\on the occasion of his 73rd birthday, in gratitude for his guidance and friendship }
\end{center}

\sloppy
\vspace*{.15cm}
\begin{abstract}
B.-Y. Chen initiated the study of warped product submanifolds in his fundamental seminal papers \cite{C1,C2,C2.1}. In this paper, we study contact CR-warped  product submanifolds of cosymplectic space forms and prove an optimal inequality by using Gauss and Codazzi equations. In addition, we obtain two geometric inequalities for contact CR-warped product submanifolds with a compact invariant factor.

\end{abstract}

\section{Introduction}
Let $(M_1, g_1)$ and $(M_2, g_2)$ be two Riemannian manifolds and $f:M_1\to(0, \infty)$ and $\pi_1:M_1\times M_2\to M_1,~\pi_2:M_1\times M_2\to M_2$, the projections map given by $\pi_1(p, q)=p$ and $\pi_2(p, q)=q$  for any $(p, q)\in M_1\times M_2$. Then, the warped product $M=M_1\times_fM_2$ is the product manifold $M_1\times M_2$ equipped with the Riemannian structure such that
\begin{align}
\label{1}
g(X, Y)=g_1(\pi_{1}{*}X, \pi_{1}{*}Y)+(f\circ\pi_1)^2g_2(\pi_{2}{*}X, \pi_{2}{*}Y)
\end{align}
for any $X, Y$ tangent to $M$, where $*$ is the symbol for the tangent maps. The function $f$ is called the warping function of $M$. In particular a warped product manifold is said to be {\it{trivial}} or Riemannian product manifold if the warping function is constant. 

 Let $M=M_1\times _fM_2$ be a warped product. For a vector field $X$ tangent to $M_1$, the lift of $X$ on $M=M_1\times _fM_2$ is the tangent vector field $\tilde X$ on $M=M_1\times _fM_2$ whose value at each $(p, q)$ is the lift $X_p$ to $(p, q)$. Thus, the lift of $X$ is the unique vector field on $M=M_1\times _fM_2$ that is $\pi_{M_1}$-related to $X$ and $\pi_{M_2}$-related to the zero vector field on $M_2$. The set of all such lifts of vector fields on $M_1$ is denoted by $\mathcal{L}(M_1)$. Similarly, we denote by $\mathcal{L}(M_2)$ the lifts of vector fields from vector fields tangent to $M_2$.
 
 Then for unit vector fields $X, Y\in\mathcal{L}(M_1)$ and $Z\in\mathcal{L}(M_2)$, we have
\begin{align}
\label{2}
\nabla_XZ=\nabla_ZX=X(\ln f)Z,\,\,\,\,g(\nabla_XY, Z)=0
\end{align}
which implies that (\cite{O}, page 210)
\begin{align}
\label{3}
K(X\wedge Z)=\frac{1}{f}\big\{(\nabla_XX)f-X^2f\big\}.
\end{align}
If we choose the local orthonormal frame $e_1,\cdots,e_n$ such that $e_1,\cdots, e_{n_1}$ are tangent to $M_1$ and $e_{n_1+1},\cdots,e_n$ are tangent to $M_2$, then we have
\begin{align}
\label{4}
\frac{\Delta f}{f}=\sum_{i=1}^{n_1}K(e_i\wedge e_j)
\end{align}
for each $j=n_1+1,\cdots, n$. For the most up-to-date survey on warped product manifolds and submanifolds, we refer to B.-Y. Chen's books \cite{book,book17} and his survey article \cite{C5}.

Recently, M.-I. Munteanu established an inequality in  \cite{Mun} for the squared norm of the second fundamental form of a contact CR-warped product submanifold in Sasakian space form along a similar line of B.-Y. Chen \cite{C3,C4}. Further, a similar inequality has been obtained for contact CR-warped products in Kenmotsu space forms  by Arslan et al. in \cite{Ar}. On the other hand,  warped product submanifolds of cosymplectic manifolds were studied in  \cite{KK}, \cite{MK} and \cite{U2,U3,U4}. 

Motivated by these work done in this spirit, we establish in this paper  the following inequality.\\

\begin{theorem} Let $\tilde M(c)$ be a $(2m+1)$-dimensional cosymplectic space form with constant sectional curvature $c$ and $M=M_T\times _fM_\perp$ be a warped product submanifold of $\tilde M(c)$. Then we have
\begin{enumerate}
\item[{(i)}] The squared norm of the second fundamental form $\sigma$ of $M$ satisfies
\begin{align}
\label{5}
\|\sigma\|^2\geq 2q\left(\|\nabla(\ln f)\|^2-\Delta(\ln f)+\frac{pc}{2}\right),
\end{align}
where $\dim M_T=2p+1,\,\dim M_\perp=q$ and $\nabla(\ln f)$ is the gradient of $\ln f$ and $\Delta$ is the Laplacian operator of $M_T$.
\item[{(ii)}] If the equality sign holds in (i), then $M_T$ is totally geodesic in $\tilde M(c)$ and $M_\perp$ is a totally umbilical submanifold of $\tilde M(c)$.
\end{enumerate}
\end{theorem}

The paper is organized as follows: Section 2 is devoted to preliminaries. In Section 3, first we develop some basic results for later use and then we prove Theorem 1.1. In the last section, we prove two geometric inequalities  as an application of Theorem 1.1 by considering compact invariant factor $M_T$.

\section{Preliminaries}
Let $\tilde M$ be a $(2m+1)$-dimensional almost contact manifold with almost contact structure $(\varphi,\xi, \eta)$, i.e., a structure vector field $\xi$, a $(1, 1)$ tensor field $\varphi$ and a $1$-form $\eta$ on $\tilde M$ such that $\varphi^2X=-X+\eta(X)\xi,~\eta(\xi)=1$, for any vector field $X$ on $\tilde M$ \cite{BL}. There always exists a compatible Riemannian metric $g$ satisfying $g(\varphi X,\varphi Y)=g(X, Y)-\eta(X)\eta(Y)$, for any vector field $X, Y$ tangent to $\tilde M$. Thus the manifold $\tilde M$ is said to be almost contact metric manifold and $(\varphi, \xi, \eta, g)$ is its almost contact metric structure. It is clear that $\eta(X)=g(X, \xi)$. The fundamental 2-form $\Phi$ on $\tilde M$ is defined as $\Phi(X, Y)=g(X, \varphi Y)$, for any vector fields $X,~Y$ tangent to $\tilde M$. The manifold $\tilde M$ is said to be {\it{almost cosymplectic}} if the forms $\eta$ and $\Phi$ are closed, i.e., $d\eta=0$ and $d\Phi=0$, where $d$ is an exterior differential operator. An almost cosyplectic and normal manifold is {\it{cosymplectic}}. It is well known that an almost contact metric manifold $\tilde M$ is cosymplectic if and only if $\tilde\nabla_X\varphi$ vanishes identically, where $\tilde\nabla$ is the Levi-Civita connection on $\tilde M$ \cite{L}.

A cosymplectic manifold $\tilde M$ with constant $\varphi$-sectional curvature is called a {\it{cosymplectic space form}} and denoted by $\tilde M(c)$. Then the Riemannian curvature tensor $\tilde R$ is given by 
\begin{align}
\label{6}
\tilde R(X, Y; Z, W)=&\frac{c}{4}\big\{g(X, W )g(Y, Z)-g(X, Z)g(Y, W )+g(X, \varphi W )g(Y, \varphi Z)\notag\\
&-g(X, \varphi Z)g(Y, \varphi W )-2g(X, \varphi Y )g(Z, \varphi W)\notag\\
&-g(X, W )\eta(Y )\eta(Z)+g(X, Z)\eta(Y )\eta(W )\notag\\
&-g(Y, Z)\eta(X)\eta(W )+g(Y, W )\eta(X)\eta(Z)\big\}.
\end{align}

Let $M$ be a $n$-dimensional Riemannian manifold isometrically immersed in a Riemannian manifold $\tilde M$. Then, the Gauss and Weingarten formulae are respectively given by $\tilde\nabla_X Y=\nabla_X Y+\sigma(X,Y)$ and $\tilde\nabla_XN=-A_NX+\nabla^\perp_XN$, for any $X,Y$ tangent to $M$, where $\nabla$ is the induced Riemannian connection on $M$, $N$ is a vector field normal to $M$, $\sigma$ is the second fundamental form of $M$, $\nabla^\perp$ is the
normal connection in the normal bundle $ TM^\perp$ and $A_N$ is the shape operator of the second fundamental form. They are related as $g(A_NX,Y)=g(\sigma(X,Y),N)$, where $g$ denotes the Riemannian metric on $\tilde M$ as well as the metric induced on $M$. 

Let $M$ be an $n$-dimensional submanifold of an almost contact metric $(2m+1)$-manifold $\tilde M$ such that restricted to $M$, the vectors $e_1,\cdots, e_n$ are tangent to $M$ and hence $e_{n+1},\cdots e_{2m+1}$ are normal to $M$. Then, the mean curvature vector $\vec H$ is defined by $\vec{H}=\frac{1}{n} tr\sigma=\frac{1}{n}\sum_{i,j=1}^n\sigma(e_i, e_i)$, where $\{e_1,\cdots, e_n\}$ is a local orthonormal frame of the tangent bundle $TM$ of $M$. A submanifold $M$ is called minimal in $\tilde M$ if its mean curvature vector vanishes identically and $M$ is totally geodesic in $\tilde M$, if $\sigma(X, Y)=0$, for all $X, Y$ tangent to $M$. If $\sigma(X,Y) =g(X,Y)H$ for all $X, Y$ tangent to $M$, then $M$ is {\it totally umbilical} submanifold of $\tilde M$. 

For any $X$ tangent to $M$, we decompose $\varphi X$ as $\phi X=PX+FX$, where $PX$ and $FX$ are the tangential and normal components of $\varphi X$, respectively. For a submanifold $M$ of an almost contact manifold $\bar M$, if $F$ is identically zero then $M$ is $invariant$ and if $P$ is identically zero then $M$ is {\it{anti-invariant}}.

Let $R$ and $\tilde R$ denote the Riemannian curvature tensors of $M$ and $\tilde M$, respectively. Then the equation of Gauss is given by
\begin{align}
\label{7}
R(X,Y;Z,W)=&\tilde R(X,Y,Z,W)+g(\sigma(X, W), \sigma(Y, Z))\notag\\
&-g(\sigma(X, Z), \sigma(Y, W)),
\end{align}
for $X, Y, Z, W$ tangent to $M$.

For the second fundamental form $\sigma$, we define the covariant derivative $\tilde\nabla\sigma$ by
\begin{align}
\label{8}
(\tilde\nabla_X\sigma)(Y, Z)=\nabla^\perp_X\sigma(Y, Z)-\sigma(\nabla_XY, Z)-\sigma(Y, \nabla_XZ)
\end{align}
for any $X, Y, Z$ tangent to $M$.

\noindent The equation of Codazzi is
\begin{align}
\label{9}
(\tilde R(X, Y)Z)^\perp=(\tilde\nabla_X\sigma)(Y, Z)-(\tilde\nabla_Y\sigma)(X, Z),
\end{align}
where $(\tilde R(X, Y)Z)^\perp$ is the normal component of $(\tilde R(X, Y)Z)$.

Let $M$ be a Riemannian $p$-manifold and $e_1,\cdots, e_p$ be an orthonormal frame fields on $M$. Then for a differentiable function $\psi$ on $M$, the Laplacian $\Delta\psi$ of $\psi$ is defined by
\begin{align}
\label{10}
\Delta\psi=\sum_{i=1}^n\big\{(\tilde\nabla_{e_i}e_i)\psi-e_ie_i\psi\big\}.
\end{align}
The scalar curvature of $M$ at a point $p$ in $M$ is given by
\begin{align}
\label{11}
\tau(p)=\sum_{1\leq i<j\leq n}K(e_i, e_j),
\end{align}
where $K(e_i, e_j)$ denotes the sectional curvature of the plane section spanned by $e_i$ and $e_j$.

Due to behaviour of the tensor field $\varphi$, there are different classes of submanifolds. We mention the following:

\begin{enumerate}
\item [(1)] A submanifold $M$ tangent to the structure vector field $\xi$ is called an invariant submanifold if $\varphi$ preserves any tangent space of $M$, i.e., $\varphi(T_pM)\subseteq T_pM$, for each $p\in M$.
\item [(2)] A submanifold $M$ tangent to the structure vector field $\xi$ is said to be an anti-invariant submanifold if $\varphi$ maps any tangent space of $M$ into the normal space, i.e., $\varphi(T_pM)\subseteq T_pM^\perp$, for each $p\in M$.
\item [(3)] A submanifold $M$ tangent to the structure vector field $\xi$ is called a contact CR-submanifold if it admits an invariant distribution $\mathcal{D}$ whose orthogonal complementary distribution $\mathcal{D}^\perp$ is anti-invariant, i.e., the tangent space of $M$ is decomposed as $TM=\mathcal{D}\oplus\mathcal{D}^\perp\oplus\langle\xi\rangle$ with $\varphi\mathcal{D}_p\subseteq \mathcal{D}_p$ and  $\varphi\mathcal{D}_p^\perp\subseteq T_pM^\perp$, for each $p\in M$, where $\langle\xi\rangle$ denotes the 1-dimensional distribution spanned by the structure vector field $\xi$.
\end{enumerate}

In this paper we study contact CR-warped product submanifolds, therefore we are concerned with the case (3). For a contact CR-submanifold $M$  of an almost contact metric manifold $\tilde M$, the normal bundle $TM^\perp$ is decomposed as
\begin{align}
\label{12}
TM^\perp=\varphi\mathcal{D}^\perp\oplus\mu,~~~\varphi\mathcal{D}^\perp\perp\mu
\end{align}
where $\mu$ is orthogonal complementary distribution of $\varphi\mathcal{D}^\perp$ which invariant normal subbundle of $TM^\perp$ with respect to $\varphi$.

\section{Proof of Theorem 1.1}
Before proving the theorem we need some basic results for warped products. A warped product submanifold $M$ is said to be a contact CR-warped product submanifold if $M=M_T\times_fM_\perp$ is the product of $M_T$ and $M_\perp$, where $M_T$ is an invariant submanifold and $M_\perp$ is an anti-invariant submanifold. We need the following results to prove our main Theorem 1.1.\\

\begin{lemma} \cite{U3} Let $M=M_T\times{_{f}M_\perp}$ be a contact CR-warped product submanifold of a cosymplectic manifold $\tilde M$. Then
\begin{align}
\label{13}
g(\sigma(X, Z), \varphi W)=-\varphi X(\ln f)g(Z, W),
\end{align}
for any $X\in\mathcal{L}(M_T)$ and $Z,~W\in\mathcal{L}(M_\perp)$.
\end{lemma}

\begin{theorem} \cite{U5} Let $M=M_T\times{_{f}M_\perp}$ be a contact CR-warped product submanifold of a cosymplectic manifold $\tilde M$. Then
\begin{enumerate}
\item [(i)] The squared norm of the second fundamental form of $M$ satisfies
\begin{align}
\label{14}
\|\sigma\|^2\geq 2q\|\nabla(\ln f)\|^2
\end{align}
where $\nabla(\ln f)$ is gradient of the function $\ln f$ and $q$ is the dimension of $M_\perp$.

\item [(ii)] If equality holds in \eqref{14}, then $M_T$ is totally geodesic submanifold of $\tilde M$ and $M_\perp$ is totally umbilical in $\tilde M$ and hence $M$ is minimal in $\tilde M$.
\end{enumerate}
\end{theorem}

\noindent Let $M=M_T\times_fM_\perp$ be a contact CR-warped product submanifold of a cosymplectic manifold $\tilde M$ such that  $\xi\in\mathcal{L}(M_T)$, where  $M_T$ and  $M_\perp$ are invariant and anti-invariant submanifolds of $\tilde M$, respectively. For any $X\in\mathcal{L}(M_T)$ and any $Z\in\mathcal{L}(M_\perp)$, we have
\begin{align*}
\sigma(\varphi X, Z)=\tilde\nabla_Z\varphi X-\nabla_Z\varphi X.
\end{align*}
Using the covariant derivative property of $\tilde\nabla\varphi$ and the cosymplectic characteristic equation with \eqref{2}, we derive
\begin{align}
\label{15}
\sigma(\varphi X, Z)&=\varphi\tilde\nabla_ZX-\varphi X(\ln f)Z\notag\\
&=\varphi\sigma(X, Z)+X(\ln f)\varphi Z-\varphi X(\ln f)Z.
\end{align}

\noindent{\it{Proof of Theorem 1.1.}}  Let $M=M_T\times_fM_\perp$ be a contact CR-warped product submanifold of a cosymplectic space form $\tilde M(c)$ such that  $\xi$ is tangential to $M_T$. Then, for the unit normal vectors $X\in\mathcal{L}(M_T)$ and $Z\in\mathcal{L}(M_\perp)$, from \eqref{9}, we have
\begin{align*}
\tilde R(X, \varphi X, Z, \varphi Z)=g((\nabla_X^\perp\sigma)(\varphi X, Z), \varphi Z)-g((\nabla_{\varphi X}^\perp\sigma)(X, Z), \varphi Z).
\end{align*}
Then from \eqref{8}, we derive
\begin{align}
\label{16}
\tilde R(X, \varphi X, Z, \varphi Z)&=g(\nabla_X^\perp\sigma(\varphi X, Z)-\sigma(\nabla_X\varphi X, Z)-\sigma(\varphi X, \nabla_XZ), \varphi Z)\notag\\
&-g(\nabla_{\varphi X}^\perp\sigma(X, Z)-\sigma(\nabla_{\varphi X}X, Z)-\sigma(X, \nabla_{\varphi X}Z), \varphi Z).
\end{align}
Now, we compute the following terms as follows
\begin{align*}
g(\nabla_X^\perp\sigma(\varphi X, Z), \varphi Z)&=Xg(\sigma(\varphi X, Z), \varphi Z)-g(\sigma(\varphi X, Z), \nabla_X^\perp\varphi Z)\\
&=Xg(\tilde\nabla_Z\varphi X, \varphi Z)-g(\sigma(\varphi X, Z), \tilde\nabla_X\varphi Z).
\end{align*}
Using the cosymplectic characteristic equation and the compatible metric property and the fact that $\xi$ is tangent to $M_T$, we derive
\begin{align*}
g(\nabla_X^\perp\sigma(\varphi X, Z), \varphi Z)=Xg(\tilde\nabla_ZX, Z)-g(\sigma(\varphi X, Z), \varphi\tilde\nabla_XZ).
\end{align*}
Then by Gauss formula and the relation \eqref{2}, we obtain
\begin{align*}
g(\nabla_X^\perp\sigma(\varphi X, Z), \varphi Z)&=X(X(\ln f))g(Z, Z)-X(\ln f)g(\sigma(\varphi X, Z), \varphi Z)\\
&-g(\sigma(\varphi X, Z), \varphi\sigma(X, Z)).
\end{align*}
Then from \eqref{13} and \eqref{15} , we get
\begin{align*}
g(\nabla_X^\perp\sigma(\varphi X, Z), \varphi Z)&=X^2(\ln f)g(Z, Z)+2X(\ln f)g(\nabla_XZ, Z)\\
&-(X(\ln f))^2g(Z, Z)-\|\sigma(X, Z)\|^2-\varphi X(\ln f)g(\sigma(X, Z), \varphi Z).
\end{align*}
Again using \eqref{2} and \eqref{13}, we derive
\begin{align}
\label{17}
g(\nabla_X^\perp\sigma(\varphi X, Z), \varphi Z)&=X^2(\ln f)g(Z, Z)+(X(\ln f))^2g(Z, Z)\notag\\
&-\|\sigma(X, Z)\|^2+(\varphi X(\ln f))^2g(Z, Z).
\end{align}
Since $M_T$ is invariant and totally geodesic in $M_T\times_fM_\perp$ \cite{Bi,C1}, then by the cosymlectic characteristic equation, we have
\begin{align}
\label{18}
g(\sigma(\nabla_X\varphi X, Z), \varphi Z)=g(\sigma(\varphi\nabla_XX, Z), \varphi Z).
\end{align}
Also, from \eqref{13}, we have $g(h(\varphi X, Z), \varphi W)=X(\ln f)g(Z, W)$, thus with the help of this fact \eqref{18} becomes
\begin{align}
\label{19}
g(\sigma(\nabla_X\varphi X, Z), \varphi Z)=(\nabla_XX\ln f)g(Z, Z).
\end{align}
Similarly, we obtain the following
\begin{align*}
g(\sigma(\varphi X, \nabla_XZ), \varphi Z)=(X\ln f)g(\sigma(\varphi X, Z), \varphi Z).
\end{align*}
Then from \eqref{13}, we get
\begin{align}
\label{20}
g(\sigma(\varphi X, \nabla_XZ), \varphi Z)=(X\ln f)^2g(Z, Z).
\end{align}
Now, interchanging $X$ by $\varphi X$ in \eqref{17}, \eqref{19} and \eqref{20}, then the following relations hold respectively
\begin{align}
\label{21}
-g(\nabla_{\varphi X}^\perp\sigma(X, Z), \varphi Z)&=((\varphi X)^2\ln f)g(Z, Z)+(\varphi X\ln f)^2 g(Z, Z)\notag\\
&-\|\sigma(\varphi X, Z)\|^2+(X\ln f)^2g(Z, Z),
\end{align}
\begin{align}
\label{22}
g(\sigma(\nabla_{\varphi X}X, Z), \varphi Z)=-(\nabla_{\varphi X}\varphi X\ln f)g(Z, Z)
\end{align}
and
\begin{align}
\label{23}
g(\sigma(X, \nabla_{\varphi X}Z), \varphi Z)=-(\varphi X\ln f)^2g(Z, Z).
\end{align}
Now, we compute 
\begin{align*}
\|\sigma(\varphi X, Z)\|^2=g(\sigma(\varphi X, Z), \sigma(\varphi X, Z)).
\end{align*}
Then from \eqref{15}, we derive
\begin{align*}
\|\sigma(\varphi X, Z)\|^2&=g(\varphi \sigma(X, Z), \varphi \sigma(X, Z))+2(\varphi X\ln f)g(\sigma(X, Z), \varphi Z)\\
&+(X\ln f)^2g(Z, Z)+(\varphi X\ln f)^2g(Z, Z).
\end{align*}
Then by the property of compatible metric and \eqref{13}, we obtain
\begin{align}
\label{24}
\|\sigma(\varphi X, Z)\|^2=\|\sigma(X, Z)\|^2-(\varphi X\ln f)^2g(Z, Z)+(X\ln f)^2g(Z, Z).
\end{align}
With the help of \eqref{24}, the relation \eqref{21} becomes
\begin{align}
\label{25}
g(\nabla_{\varphi X}^\perp\sigma(X, Z), \varphi Z)=&-((\varphi X)^2\ln f)g(Z, Z)+\|\sigma(X, Z)\|^2\notag\\
&-2(\varphi X\ln f)^2g(Z, Z).
\end{align}
Then from \eqref{16},  \eqref{17},  \eqref{19},  \eqref{20} and  \eqref{25}, we derive
\begin{align}
\label{26}
\tilde R(X, \varphi X, Z, \varphi Z)&=(X^2\ln f)g(Z, Z)-(\nabla_XX\ln f)g(Z, Z)\notag\\
&+((\varphi X)^2\ln f)g(Z, Z)-(\nabla_{\varphi X}\varphi X\ln f)g(Z, Z)\notag\\
&-2\|\sigma(X, Z)\|^2+2(\varphi X\ln f)^2g(Z, Z).
\end{align}
Also, from \eqref{6} and \eqref{7}, we have
\begin{align}
\label{27}
\tilde R(X, \varphi X, Z, \varphi Z)=\frac{c}{2}\big(\eta^2(X)-\|X\|^2\big)g(Z, Z),
\end{align}
for any unit tangent vector $X\in\mathcal{L}(M_T)$ and $Z\in\mathcal{L}(M_\perp)$. Then from \eqref{26} and \eqref{27}, we get
\begin{align}
\label{28}
2\|\sigma(X, Z)\|^2&=(X^2\ln f)g(Z, Z)-(\nabla_XX\ln f)g(Z, Z)+((\varphi X)^2\ln f)g(Z, Z)\notag\\
&-(\nabla_{\varphi X}\varphi X\ln f)g(Z, Z)+2(\varphi X\ln f)^2g(Z, Z)\notag\\
&-\frac{c}{2}\big(\eta^2(X)-\|X\|^2\big)g(Z, Z).
\end{align}
Now, consider the orthonormal frame fields of $M_T$ and $M_\perp$ as follows: $\{X_1,\cdots, X_p, X_{p+1}=\varphi X_1,\cdots, X_{2p}=\varphi X_p, X_{2p+1}=\xi\}$ and $\{Z_1,\cdots, Z_q\}$ are the frame fields of the tangent spaces of $M_T$ and $M_\perp$, then summing over $i=1,\cdots,2p+1$ and $j=1,\cdots, q$ in \eqref{28}, thus we have
\begin{align*}
2\sum_{i=1}^{2p+1}\sum_{j=1}^{q}\|\sigma(X_i, Z_j)\|^2&=-\sum_{i=1}^{2p+1}\sum_{j=1}^{q}\left((\nabla_X{_i}X_i(\ln f))-X_i^2(\ln f)\right)g(Z_j, Z_j)\\
&-\sum_{i=1}^{2p+1}\sum_{j=1}^{q}\left((\nabla_{\varphi X_i}\varphi X_i(\ln f))-(\varphi X_i)^2(\ln f)\right)g(Z_j, Z_j)\\
&+2\sum_{i=1}^{2p+1}\sum_{j=1}^{q}(\varphi X_i(\ln f))^2g(Z_j, Z_j)\\
&-\sum_{i=1}^{2p+1}\sum_{j=1}^{q}\frac{c}{2}\big(\eta^2(X_i)-\|X_i\|^2\big)g(Z_j, Z_j).
\end{align*}
Then, from the definition of the gradient and \eqref{10}, we derive
\begin{align*}
2\sum_{i=1}^{2p+1}\sum_{j=1}^{q}\|\sigma(X_i, Z_j)\|^2=-2q\Delta(\ln f)+2q\|\varphi\nabla(\ln f)\|^2+pqc
\end{align*}
or
\begin{align}
\label{29}
\|\sigma\|^2_{T\perp}=-q\Delta(\ln f)+q\|\nabla(\ln f)\|^2+\frac{pqc}{2}\notag\\
\end{align}
Thus, from Theorem 1.1 of  \cite{U5}, we have $\|\sigma\|^2\geq2\|\sigma\|^2_{T\perp}$ only and left all other terms in the right hand side in the inequality of that theorem. Then using \eqref{29} in this relation we get inequality (i) with equality sign holds if and only if
\begin{align}
\label{30}
\sigma(\mathcal{L}(M_T), \mathcal{L}(M_T))=\sigma(\mathcal{L}(M_\perp), \mathcal{L}(M_\perp))=0.
\end{align}
i.e., $M$ is both $M_T$ and $M_\perp$-totally geodesic. The equality case holds just like Theorem 3.1. Hence, the proof is complete.

\section{Some Applications of Theorem 1.1}
For a warped product CR-submanifold $M_T\times_fM_\perp$ of a cosymplectic space form, if the holomorphic submanifold $M_T$ is compact, then we have the following useful results.

\begin{theorem} Let $M_T\times_fM_\perp$ be a warped product CR-submanifold of a cosymplectic space form $\tilde M(c)$ such that $M_T$ is a compact invariant submanifold of $\tilde M(c)$. Then
\begin{enumerate}
\item[{(i)}] For any $s\in M_\perp$, we have
\begin{align}
\label{31}
\int_{M_T\times\{s\}}\|\sigma\|^2\,dV_T\geq pqc\,\rm{vol}(M_T),
\end{align}
where $dV_T$ and $\rm{vol}(M_T)$ are the volume element and the volume of $M_T$, respectively and $2p+1=\dim M_T,\,q=\dim M_\perp$.
\item[{(ii)}] The equality sign holds in (i) identically if and only if $M$ is a Riemannian product of $M_T$ and $M_\perp$, i.e., the warping function $f$ is constant on $M$ .
\end{enumerate}
\end{theorem}
\begin{proof} For a warped product submanifold $M_T\times_fM_\perp$  with compact $M_T$, from Theorem 1.1, we have
\begin{align}
\label{32}
\int_{M_T\times\{s\}}\|\sigma\|^2\,dV_T\geq 2q\int_{M_T\times\{s\}}\left(\|\nabla(\ln f)\|^2-\Delta(\ln f)+\frac{pc}{2}\right)\,dV_T.
\end{align}
Since $M_T$ is compact, it follows from Hopf's Lemma that
\begin{align}
\label{33}
\int_{M_T\times\{s\}}\|\sigma\|^2\,dV_T\geq 2q\int_{M_T\times\{s\}}\left(\|\nabla(\ln f)\|^2\right)\,dV_T+pqc\,\rm{vol}(M_T).
\end{align}
Thus, inequality \eqref{33} implies inequality \eqref{31}, with the equality sign holding if and only if  (1) $f$ is constant i.e., $M$ is Riemannian product and (2) the equality $\|\sigma\|^2=pqc$ holds identically. Hence, the theorem is proved completely.
\end{proof}

Next, let us assume $f$ is non-constant. Then the minimal principle on $\lambda_1$ yields (see \cite{Be} page 186, \cite{C4})
\begin{align}
\label{34}
\int_{M_T}\|\nabla(\ln f)\|^2\,dV_T\geq \lambda_1\int_{M_T}\left(\ln f\right)^2\,dV_T
\end{align}
with equality holding if and only if $\Delta\ln f=\lambda_1\ln f$ holds. 

Now, we give the following result.

\begin{theorem} Let $M=M_T\times_fM_\perp$ be a warped product CR-submanifold of a cosymplectic space form $\tilde M(c)$ with compact $M_T$. If the warping function $f$ is non-constant, then, for each $s\in M_\perp$, we have
\begin{align}
\label{35}
\int_{M_T\times\{s\}}\|\sigma\|^2\,dV_T\geq 2q\lambda_1\int_{M_T}\left(\ln f\right)^2\,dV_T+pqc\,\rm{vol}(M_T)
\end{align}
where $dV_T$, $\lambda_1$ and $\rm{vol}(M_T)$ are the volume element, the first positive eigenvalue of the Laplacian $\Delta$ and the volume of $M_T$, respectively.\\

\noindent Moreover, the equality sign of \eqref{35} holds identically if and only if we have:
\begin{enumerate}
\item[(i)] $\Delta\ln f=\lambda_1\ln f$
\item[(ii)] $M$ is both $M_T$-totally geodesic and $M_\perp$-totally geodesic.
\end{enumerate}
\end{theorem}
\begin{proof} By combining \eqref{33} and \eqref{34}, we get inequality \eqref{35}. From the above discussion, we know that the equality sign of \eqref{35} holds identically if and only if we have (i) $\Delta\ln f=\lambda_1\ln f$ and (ii) the warped product is both $M_T$ and $M_\perp$ totally geodesic.
\end{proof}

Another motivation of Theorem 1.1 is to give the expression of Dirichlet energy of the warping function in physics, which is defined of a function $\psi$ on a compact manifold $M$ as follows
\begin{align}
\label{36}
E(\psi)=\frac{1}{2}\int_{M}\|\nabla\psi\|^2\,dV
\end{align}
where $\nabla\psi$ is the gradient of $\psi$ and $dV$ is the volume element.

Now, we give the expression of Dirichlet energy of the warping function for a contact CR-warped product $M_T\times_fM_\perp$ in a cosymplectic space form $\tilde M(c)$ with compact invariant submanifold $M_T$. For any $s\in M_\perp$, from Theorem 1.1, we have 
\begin{align}
\label{37}
\int_{M_T\times\{s\}}\|\sigma\|^2\,dV_T\geq2q\int_{M_T\times\{s\}}\left(\|\nabla\ln f\|^2+\frac{pc}{2}\right)\,dV_T.
\end{align}
From \eqref{36} and \eqref{37}, we find 
\begin{align*}
E(\ln f)\leq\frac{1}{4q}\int_{M_T\times\{s\}}\|\sigma\|^2\,dV_T-\frac{pc}{4}\,\rm{vol}(M_T)
\end{align*}
which is the Dirichlet energy $E(\ln f)\,(0\leq E(\ln f)<\infty)$ of the warping function.\\

\noindent{\bf{Acknowledgements.}} The authors would like to express their hearty thanks to anonymous referees for their valuable suggestions and comments towards the improvement of this manuscript, especially, Section 4.

\noindent Author's addresses:\\

\noindent Falleh R. Al-Solamy:
 
\noindent Department of Mathematics, Faculty of Science, King Abdulaziz University, 21589 Jeddah, Saudi Arabia

\noindent {\it E-mail}: {\tt falleh@hotmail.com}\\

\noindent Siraj Uddin:

\noindent Department of Mathematics, Faculty of Science, King Abdulaziz University, 21589 Jeddah, Saudi Arabia

\noindent {\it E-mail}: {\tt siraj.ch@gmail.com}\\

\end{document}